\documentclass{amsart}
\usepackage{amssymb}
\usepackage[T1]{fontenc}
\usepackage{mathrsfs}
\usepackage{enumerate}
 \usepackage{hyperref}
 \usepackage{cleveref}

\newtheorem{thm}{Theorem}
\newtheorem{defn}{Definition}
\newtheorem{fact}{Fact}
\newtheorem{lem}[thm]{Lemma}

\newtheorem{cor}[thm]{Corollary}
\newtheorem{question}{Question}

\newcommand{\pp}{\mathcal{P}}
\newcommand{\up}{\upharpoonright}
\newcommand{\ra}{\rightarrow}
\newcommand{\mc}{\mathcal}
\newcommand{\ms}{\mathscr}
\newcommand{\msa}{\mathscr{A}}
\newcommand{\msb}{\mathscr{B}}
\newcommand{\mcf}{\mathcal{F}}
\newcommand{\mcx}{\mathcal{X}}

\newcommand{\lo}{<_{ho}}

\makeatletter
\@addtoreset{case}{thm}
\@addtoreset{case}{lem}
\@addtoreset{case}{prop}
\makeatother

\begin{document}

  \title{A ccc indestructible construction with CH}
  \author{Yinhe Peng}
  \address{Academy of Mathematics and Systems Science, Chinese Academy of Sciences\\ East Zhong Guan Cun Road No. 55\\Beijing 100190\\China}
\email{pengyinhe@amss.ac.cn}

\thanks{The author was partially supported by NSFC No. 12571001.}


  \begin{abstract}
  We introduce a variant of the Kurepa family. We then use one such  family to construct a ccc indestructible property associated with a complete coherent Suslin tree $S$. Moreover, in every ccc forcing extension that preserves Suslin of $S$, forcing with $S$ induces a strong negative partition relation.
  
  \end{abstract}
  
    \maketitle

  \section{Introduction}
  
 When dealing with a specific problem, it is common to construct, for every uncountable subset, an uncountable subset with some additional property. For example, to destroy Suslin subtrees of an Aronszajn tree $T$, we need to add an uncountable antichain to every uncountable subset of $T$; to add precaliber $\omega_1$  (property K$_n$) to a ccc poset $\mc{P}$, we need to add an uncountable centered  ($n$-linked) subset to every uncountable subset of $\pp$.
  

A poset $\pp$ has \emph{precaliber $\omega_1$} if every uncountable subset of $\pp$ has an uncountable centered subset. For $n\geq 2$, $\pp$ has \emph{property {\rm K}$_n$} if every uncountable subset of $\pp$ has an uncountable subset that is $n$-linked.   A subset $X$ of $\pp$ is \emph{$n$-linked} if every $n$-element subset of $X$ has a common lower bound. We omit $n$ if $n=2$.

There are two standard  methods to obtain sufficiently many uncountable subsets with an additional property. One method is to guarantee the property by the structure. For example, the random forcing has property K$_n$ for all $n\geq 2$. The other method is to force sufficiently many uncountable subsets via iterated forcing (see, e.g., \cite{PW24}).

However, the feature of each method may limit its usage in specific problems: the property obtained by the first method is in general absolute, i.e., the property cannot be destroyed by forcing; and the property obtained by the second method usually exists in models with large continuum.

We will introduce a variant of the Kurepa family. This variant is consistent with CH (or $\diamondsuit^+$) and has potential to construct, e.g., (non-absolute) posets with precaliber $\omega_1$ (see Definition \ref{defn non-absolute} and the paragraph below). As an example, we use one such family to construct a coloring on pairs of a complete coherent Suslin tree that has the   property achieved via iterated forcing in \cite{PW24}. Moreover, colorings on pairs of a Suslin tree have interests of their own.

Strong negative partition relations on $\omega_1$ play important roles in constructing structures of size $\omega_1$ in ZFC. In \cite{Todorcevic87}, Todorcevic introduced the minimal walk techique and proved  $\omega_1\not\ra [\omega_1]^2_{\omega_1}$. In \cite{Moore06}, Moore employed a new structure to prove a general form of strong negative partition relation which implies the existence of an L space. In \cite{PW18}, the author and Wu proved the finite dimensional version Pr$_1(\omega_1, \omega, n+1)$ for $n>1$ and its general form which implies the existence of an L space whose $n$th power is still an L space.

 For $\theta\leq \omega$, Pr$_1(\omega_1, \kappa, \theta)$ is the following statement:
\begin{itemize}
  \item (\cite{Galvin},\cite{Shelah94}) There is a function $c:[\omega_1]^2\rightarrow \kappa$ such that whenever we are given $n<\theta$, a collection $\langle a_\alpha: \alpha< \omega_1\rangle$ of pairwise disjoint elements of $[\omega_1]^{n}$ and  $\eta<\kappa$,   there are $\alpha< \beta$ such that $c( a_\alpha(i), a_\beta(j))=\eta$ for all $i, j<n$.
  \end{itemize}
  Here $a(i)$ is the $i$th element of $a$ in the increasing enumeration for $a\in [\omega_1]^n$ and $i<n$.
 It turns out that  Pr$_1(\omega_1, \omega, n)$ for all $n>1$ is optimal according to
Galvin's observation that MA$_{\omega_1}$ implies the failure of Pr$_1(\omega_1, 2, \omega)$. On the other hand,  Pr$_1(\omega_1, \omega, \omega)$ follows from $\mathfrak{b}=\omega_1$ (see \cite{Todorcevic89}).

A natural question which is implicitly asked in \cite[1.5]{Todorcevic89} is whether  Pr$_1(\omega_1, 2, \omega)$ follows from $\mathfrak{t}=\omega_1$. In other words, it is interesting to explore constructions for the structure $([\omega]^\omega, \subset^*)$  that are analogous to constructions for the structure $(\omega^\omega, <^*)$. 

A strategy towards a negative answer has been applied in other problems (see \cite{LT}) and has potential to distinguish closely related properties, e.g., MA$_{\omega_1}$ and $\mathscr{K}_2$ (see \cite{TV}). This consists in considering a forcing axiom-like statement which   implies $\mathfrak{t}=\omega_1$, e.g., MA$_{\omega_1}(S)[S]$, and proving that it implies $\neg$Pr$_1(\omega_1, 2, \omega)$. Note that for a Suslin tree $S$, $\Vdash_S \mathfrak{t}=\omega_1$.
  
  For a Suslin tree $S$, MA$_{\omega_1}(S)$ is the following assertion (see \cite{LT}):
  \begin{itemize}
  \item for any c.c.c. poset $\mathcal{P}$ with $\Vdash_\mc{P} S$ is Suslin, for any collection $\{D_\alpha: \alpha<\omega_1\}$ of dense   subsets of $\mc{P}$, there is a filter $G\subseteq \mc{P}$ meeting them all.
  \end{itemize}
MA$_{\omega_1}(S)[S]$ holds if (see \cite{Larson-Tall}) the universe is a forcing extension by $S$ over a model of MA$_{\omega_1}(S)$ where $S$ is the Suslin tree. PFA$(S)$ and PFA$(S)[S]$ are defined in an analogous way.

In \cite{PW24}, a model of $\mathrm{MA}_{\omega_1}(S)[S]$ is  constructed in which $\mathrm{Pr}_1(\omega_1, 2,\omega)$ holds. The model is constructed via iterated ccc forcing while a specific property is designed and preserved in the iteration process. This does not exclude the possibility of constructing a model of $\mathrm{MA}_{\omega_1}(S)[S]$ via iterated ccc forcing (while preserving another property) together with $\neg\mathrm{Pr}_1(\omega_1, 2,\omega)$.

We will introduce  the thin Kurepa family (Definition \ref{defn tkf}) which is ccc indestructible (Lemma \ref{lem ccc indestructible}) and consistent with GCH and $\diamondsuit^+$ (Theorem \ref{thm consistency}). Moreover, in every ccc forcing extension of a model with such families, if $\mathrm{MA}_{\omega_1}(S)[S]$ holds, then $\mathrm{Pr}_1(\omega_1, 2,\omega)$ holds (Corollary \ref{cor indes}). This suggests that, in order to get a model of $\mathrm{MA}_{\omega_1}(S)[S]+\neg\mathrm{Pr}_1(\omega_1, 2,\omega)$,  we probably need to force non-ccc posets in the iteration process.
In fact, whether the stronger property PFA$(S)[S]$ decides $\mathrm{Pr}_1(\omega_1, 2, \omega)$ remains open.
\begin{question}
Does $\mathrm{PFA}(S)[S]$ imply  $\neg\mathrm{Pr}_1(\omega_1, 2, \omega)$? Or $\ms{K}_2$?
\end{question}
$\ms{K}_2$ is the assertion that every ccc poset has property K, first considered by Knaster and Szpilrajn in Problem 192 of the Scottish Book in the 1940s (see \cite{The Scottish Book}), arises naturally when attempting to solve the famous Suslin Problem \cite{Souslin}. The notation $\ms{K}_2$ was first used in \cite{TV}. $\ms{K}_2$  is equivalent to the following property in form of colorings.
\begin{itemize}
\item For every coloring $\pi: [\omega_1]^{<\omega}\ra 2$ with $\pi^{-1}\{0\}$ downward closed, the poset $(\pi^{-1}\{0\}, \supseteq)$ is ccc iff it has property K.
\end{itemize}

This paper is organized as follows.   In Section 3, we introduce the thin Kurepa family, analyze its properties and prove its relative consistency with ZFC. In Section 4, we use a thin Kurepa family to construct a coloring on a complete coherent Suslin tree and show that the coloring witnesses Pr$_1(\omega_1, 2, \omega)$ in the Suslin extension.

  \section{preliminaries}
In this section, we introduce some standard definitions and facts (see \cite{Jech}, \cite{Kunen}).

$[X]^\kappa$ is the set of all subsets of $X$ of size $\kappa$. Moreover, if $X$ is a set of ordinals and $\alpha$ is less than the order type of $X$, then $X(\alpha)$ is the $\alpha$th element of $X$ in its increasing enumeration. In particular, if $k<\omega$ and $b\in [X]^k$, then $b(0),b(1),...,b(k-1)$ is the increasing enumeration of $b$.

A \emph{Kurepa family} is an $\mc{F}\subseteq P(\omega_1)$ such that $|\mc{F}|\geq \omega_2$ and $|\{A\cap \alpha: A\in \mc{F}\}|\leq \omega$ for all $\alpha<\omega_1$.

Suslin trees in this paper are subsets of $2^{<\omega_1}$, ordered by extension.

For a  Suslin tree $S$ and $t\in S$, the\emph{ height} of $t$ is $ht(t)=dom(t)$. In other words, $ht(t)$ is the order type of $\{s\in S: s<_S t\}$. For any $\alpha<\omega_1$, $S_\alpha=\{t\in S: ht(t)=\alpha\}$ is  the $\alpha$th level of $S$ and $S\up\alpha=\bigcup_{\beta<\alpha} S_\beta$. For $s\in S$,
\[S_{<s}=\{t: t<_S s\}.\]

 For any $s, t\in 2^{<\omega_1}$, 
$$\Delta(s,t)=\max \{\alpha\leq \min\{dom(s), dom(t)\}: s\up_\alpha = t\up_\alpha\},$$ 
$$s \wedge t=s\up_{\Delta(s,t)},$$
$$D_{s,t}=\{\xi<\min\{dom(s), dom(t)\}: s(\xi)\neq t(\xi)\}.$$

A Suslin tree $S$  is \emph{coherent} if  $D_{s,t}$ is finite for any $s, t\in S$. 

A coherent Suslin tree $S\subset 2^{<\omega_1}$   is \emph{complete} if for all $s\in S$ and all $t\in 2^{ht(s)}$, $t\in S$ whenever $|D_{s,t}|<\omega$. 

If we add a trivial restriction on a coherent Suslin tree $S$ that $\{s(\alpha): s\in S_{\alpha+1}\}=2$ for all $\alpha<\omega_1$, completeness is the requirement that 
\[\text{for $s, t$ in $S$ with }ht(s)<ht(t), ~s^\smallfrown t\up [ht(s), ht(t))\in S.\]  
In fact, adding completeness does not lose generality up to club isomorphism.
\begin{lem}
Assume $S\subset 2^{<\omega_1}$ is a coherent Suslin tree. Then there is a club $C$ such that for every $s, t$ in $S\up C=\bigcup_{\alpha\in C} S_\alpha$ with $ht(s)<ht(t)$, $s^\smallfrown t\up [ht(s), ht(t))\in S$. 
\end{lem}
\begin{proof}
Fix a continuous chain of countable elementary submodels $M_0\prec...M_\alpha\prec M_{\alpha+1}\prec...$ of $H_{\omega_2}$ containing $S$. We check that $C=\{M_\alpha\cap \omega_1: \alpha<\omega_1\}$ is as desired.

Suppose otherwise and $s, t$ witness the failure. Assume $ht(s)=M_\xi\cap \omega_1$. Choose $\overline{\alpha}\in M_\xi$ such that $D_{s,t}\subseteq \overline{\alpha}$.

By elementarity, for every $\beta>\overline{\alpha}$, there are $s_\beta>_S s\up \overline{\alpha}$, $t_\beta>_S t\up \overline{\alpha}$ such that 
\begin{enumerate}
\item $ht(t_\beta)>ht(s_\beta)\geq \beta,~ D_{s_\beta, t_\beta}\subseteq \overline{\alpha} \text{ and } s_\beta^\smallfrown t_\beta\up [ht(s_\beta), ht(t_\beta))\notin S$.
\end{enumerate}

Now fix an uncountable $\Gamma\subseteq \omega_1$ such that
\begin{enumerate}\setcounter{enumi}{1}
\item  for every $\beta<\gamma$ in $\Gamma$, $ht(t_\beta)<ht(s_\gamma)$. In particular,  $s_\gamma\up ht(t_\beta)\in S$.
\end{enumerate}
 By (1) and (2), $\{s_\beta^\smallfrown t_\beta\up [ht(s_\beta), ht(t_\beta)): \beta\in \Gamma\}$ is an antichain in $\omega^{<\omega_1}$. 

Since $s_\beta^\smallfrown t_\beta\up [ht(s_\beta), ht(t_\beta))=(s\up \overline{\alpha})^\smallfrown t_\beta\up [\overline{\alpha}, ht(t_\beta))$,   $\{t_\beta: \beta\in \Gamma\}$ is an antichain in $S$. A contradiction.
\end{proof}

\section{The thin Kurepa family}

In this section, we introduce and analyze properties of the thin Kurepa family.  We then construct a model with a thin  Kurepa family. Throughout the paper, we use the following notations.

Say $\msa\subseteq [\omega_1]^{<\omega}$ is \emph{non-overlapping} if for $a\neq b$ in $\msa$, either $\max(a)<\min(b)$ or $\max(b)<\min(a)$.

\begin{defn}
For $a\subseteq \omega_1$  and $\mc{F}\subseteq P([\omega_1]^{<\omega})$ such that every $\msa\in \mcf$ is pairwise disjoint, $cl(a, \mc{F})$ is the $\subseteq$-least subset $a^*$ of $\omega_1$ such that 
\[a\subseteq a^*\text{ and for all $\ms{A}\in \mc{F}$ and all $b\in \ms{A}$, either $b\cap a^*=\emptyset$ or }b\subseteq a^*.\]
For $a\subseteq \omega_1$, say $a$ is \emph{$\mcf$-closed} if $cl(a, \mcf)=a$. 
\end{defn}
Equivalently, $cl(a, \mc{F})$ is $\bigcup_{n<\omega} a_n$ where 
\[a_0=a\text{ and }a_{n+1}=\bigcup\{b: b\in \msa\text{ for some $\msa\in \mcf$ and } b\cap a_n\neq \emptyset\}\] 
and $a$ is $\mcf$-closed if for all $\ms{A}\in \mc{F}$ and all $b\in \ms{A}$, either $b\cap a=\emptyset$ or $b\subseteq a$.

We first verify the following properties of $cl$.
\begin{lem}\label{lem closure}
Suppose $\mc{F}\subseteq P([\omega_1]^{<\omega})$ such that every $\msa\in \mcf$ is pairwise disjoint.
\begin{enumerate}[(i)]
\item  For $a\subseteq \omega_1$, $cl(a, \mcf)=\bigcup_{\alpha\in a} cl(\{\alpha\}, \mcf)$.
\item The collection of $\mcf$-closed sets is closed under complement and arbitrary union. In particular, it is a $\sigma$-algebra.
\item $\{cl(\{\alpha\}, \mcf): \alpha<\omega_1\}$ is a partition of $\omega_1$.
\end{enumerate}
\end{lem}
\begin{proof}
(i) $cl(a, \mcf)\supseteq cl(\{\alpha\}, \mcf)$ for every $\alpha\in a$ follows from the definition. 

To see $cl(a, \mcf)\subseteq \bigcup_{\alpha\in a} cl(\{\alpha\}, \mcf)$, first note that $a\subseteq \bigcup_{\alpha\in a} cl(\{\alpha\}, \mcf)$. Fix $\ms{A}\in \mc{F}$ and   $b\in \ms{A}$ with $b\cap (\bigcup_{\alpha\in a} cl(\{\alpha\}, \mcf))\neq \emptyset$. Then $b\cap cl(\{\alpha\}, \mcf)\neq \emptyset$ for some $\alpha\in a$ and hence $b\subseteq cl(\{\alpha\}, \mcf)\subseteq \bigcup_{\alpha\in a} cl(\{\alpha\}, \mcf)$.

(ii)  Fix $a=\bigcup_{\alpha<\kappa} a_\alpha$ where $\kappa\leq \omega_1$ and each $a_\alpha$ is $\mcf$-closed. Then by (i),
\[a=\bigcup_{\alpha<\kappa} a_\alpha=\bigcup_{\alpha<\kappa} \bigcup_{\beta\in a_\alpha} cl(\{\beta\}, \mcf)=\bigcup_{\beta\in a} cl(\{\beta\}, \mcf)=cl(a, \mcf).\]
Now it suffices to fix a $\mcf$-closed set $a$ and show that $\omega_1\setminus a$ is $\mcf$-closed.

 Fix $\ms{A}\in \mc{F}$ and   $b\in \ms{A}$.
 Since $a=cl(a, \mcf)$, either $b\cap a=\emptyset$ or $b\subseteq a$. Equivalently,  either $b\subseteq \omega_1\setminus a$ or $b\cap (\omega_1\setminus a)=\emptyset$. This shows that $\omega_1\setminus a$ is $\mcf$-closed.

(iii) It suffices to show that for $\alpha, \beta<\omega_1$, either $cl(\{\alpha\}, \mcf)=cl(\{\beta\}, \mcf)$ or $cl(\{\alpha\}, \mcf)\cap cl(\{\beta\}, \mcf)=\emptyset$. By symmetry, it suffices to show either $cl(\{\alpha\}, \mcf)\subseteq cl(\{\beta\}, \mcf)$ or $cl(\{\alpha\}, \mcf)\cap cl(\{\beta\}, \mcf)=\emptyset$.
By (ii), 
\[cl(\{\alpha\}, \mcf)=(cl(\{\alpha\}, \mcf)\cap cl(\{\beta\}, \mcf))\cup (cl(\{\alpha\}, \mcf)\setminus cl(\{\beta\}, \mcf))\]
 is a partition of $cl(\{\alpha\}, \mcf)$ into two $\mcf$-closed sets. By definition of $cl(\{\alpha\}, \mcf)$, either $cl(\{\alpha\}, \mcf)=cl(\{\alpha\}, \mcf)\cap cl(\{\beta\}, \mcf)$ or $cl(\{\alpha\}, \mcf)=cl(\{\alpha\}, \mcf)\setminus cl(\{\beta\}, \mcf)$. In other words, either $cl(\{\alpha\}, \mcf)\subseteq cl(\{\beta\}, \mcf)$ or $cl(\{\alpha\}, \mcf)\cap cl(\{\beta\}, \mcf)=\emptyset$.
\end{proof}

The thin Kurepa family we are going to define is too large to consider its closed sets. So instead of consider $cl(a, \mcf)$ for a thin Kurepa family $\mcf$, we will consider $cl(a, \mcf')$ for $a\in [\omega_1]^{<\omega}$ and $\mcf'\in [\mcf]^{<\omega}$.

For $\mc{F}\subseteq P([\omega_1]^{<\omega})$ and $\alpha<\omega_1$, $\mcf\up\alpha=\{\msa\cap [\alpha]^{<\omega}: \msa\in \mcf\}$.

\begin{defn}\label{defn tkf}
$\mc{F}\subseteq P([\omega_1]^{<\omega})$ is a \emph{thin Kurepa family} if $\mcf$ consists of non-overlapping families and the following conditions hold.
\begin{enumerate}[{\rm (K1)}]
\item For every $\alpha<\omega_1$, $|\mcf\up \alpha|\leq \omega$.
\item For every uncountable pairwise disjoint family $\msa\subseteq [\omega_1]^{<\omega}$, there exists $\msb\in \mcf$ such that $|\msa\cap \msb|=\omega_1$.
\item For every $a\in [\omega_1]^{<\omega}$ and every $\mcf'\in [\mcf]^{<\omega}$, $cl(a, \mcf')$ is finite.
\end{enumerate}
\end{defn}

Before constructing a thin Kurepa family, we first analyze its properties. Property (K3) requires that elements in $\mcf$ are thin.  We show that every thin Kurepa family is a Kurepa family if we do not distinguish $\omega_1$ and $[\omega_1]^{<\omega}$.
\begin{lem}
Every thin Kurepa family has size $>\omega_1$.
\end{lem}
\begin{proof}
Suppose otherwise. For some $\kappa\leq \omega_1$, $\mcf=\{\msa^\alpha: \alpha<\kappa\}$ is a thin Kurepa family. For every $\alpha<\omega_1$, find $x_\alpha<\omega_1$ such that $\{\alpha, x_\alpha\}\notin \msa^\xi$ for all $\xi<\alpha$. Then find $\Gamma\in [\omega_1]^{\omega_1}$ such that $\msa=\{\{\alpha, x_\alpha\}: \alpha\in \Gamma\}$ is pairwise disjoint. It is clear that $|\msa\cap \msa^\xi|\leq \omega$ for all $\xi<\kappa$. This contradicts (K2).
\end{proof}

Being a thin Kurepa family is ccc indestructible.
\begin{lem}\label{lem ccc indestructible}
Suppose $\mcf$ is a thin Kurepa family and $\pp$ is a ccc poset. Then 
\[\Vdash_\pp \check{\mcf}\text{ is a thin Kurepa family}.\]
\end{lem}
\begin{proof}
Note that (K1) and (K3) are absolute. To show (K2), fix $p\in \pp$ and $\dot{\msa}$ such that
\[p\Vdash \dot{\msa}\text{ is an uncountable pairwise disjoint family contained in } [\omega_1]^{<\omega}.\]
For every $\alpha<\omega_1$, find $p_\alpha< p$ and $a_\alpha\in [\omega_1\setminus \alpha]^{<\omega}$ such that
$p_\alpha\Vdash a_\alpha\in \dot{\msa}$.
Find $\Gamma\in [\omega_1]^{\omega_1}$ such that
\[\msb=\{a_\alpha: \alpha\in \Gamma\}\text{ is pairwise disjoint}.\]
By (K2), find $\msb'\in \mcf$ such that $\msb\cap \msb'$ is uncountable. Let
\[\Gamma'=\{\alpha\in \Gamma: a_\alpha\in \msb'\}.\]
Since $\pp$ is ccc, there exists $q<p$ such that
\[q\Vdash \{\alpha\in \Gamma': p_\alpha\in \dot{G}\}\text{ is uncountable}\]
where $\dot{G}$ is the canonical name of the generic filter. Then
\[q\Vdash \dot{\msa}\cap \msb'\supseteq \{a_\alpha: \alpha\in \Gamma', p_\alpha\in \dot{G}\}\text{ is uncountable}.\]
Now (K2) follows from a density argument.
\end{proof}

As presented in Section 4, the thin Kurepa family may be used to construct  structures with properties of type precaliber $\omega_1$, i.e., every uncountable subset has an uncountable subset with some additional property. The construction at limit stages $\alpha$ may need a partition of $\alpha$ into $\omega$ intervals with each interval `almost' not affected by intervals below.\footnote{See the proof of Theorem \ref{thm coloring} for an example of applying the partition to control affection between intervals.}  We describe the partition in the following lemma.

\begin{lem}\label{lem tkf ladder}
Suppose $\mcf$ is a thin Kurepa family and $\mcf\up \alpha=\{\msa^n: n<\omega\}$ where $\alpha$ is a countable limit ordinal. Then there are an increasing sequence of ordinals $\langle \alpha_n: n<\omega\rangle$ cofinal in $\alpha$ and a sequence $\langle F_n\in [\alpha]^{<\omega}: n<\omega\rangle$  satisfying the following properties.
\begin{enumerate}[(i)]
\item $F_n$'s are pairwise disjoint.
\item For every $n<\omega$, $F_n=cl(F_n, \{\msa^i: i<n\})\subseteq (\alpha_{n-1}, \alpha_{n+1})$ and $\alpha_n\cup F_n=cl(\alpha_n\cup F_n, \{\msa^i: i<n\})$.
\item For every $n<\omega$ and every $\beta\in [\alpha_{n}, \alpha_{n+1})$, 
\[cl(\{\beta\}, \{\msa^i: i<n\})\subseteq [\alpha_n, \alpha_{n+1})\cup F_n\cup F_{n+1}.\]
\end{enumerate}
\end{lem}
\begin{proof}
Let $\langle \beta_n: n<\omega\rangle$ be an increasing sequence cofinal in $\alpha$. We inductively choose an increasing sequence $\langle \alpha_n: n<\omega\rangle$ such that $\beta_n\leq \alpha_n<\alpha$ and
\begin{enumerate}
\item  $cl(\bigcup A_n, \{\msa^i: i\leq n\})\subseteq   \alpha_{n+1}$ where
\[A_n=\{a\in \bigcup_{i\leq n}\msa^i:  \min(a)\leq \alpha_n\leq \max(a)\}.\].
\end{enumerate}
To see the existence of $\alpha_{n+1}<\alpha$, note each $\msa^i\subseteq [\alpha]^{<\omega}$ and hence $cl(\bigcup A_n, \{\msa^i: i\leq n\})\subseteq  \alpha$. Since $\msa^i$'s are non-overlapping, $A_n$ is finite. By (K3), $cl(\bigcup A_n, \{\msa^i: i\leq n\})$ is finite.  So $cl(\bigcup A_n, \{\msa^i: i\leq n\})\in [\alpha]^{<\omega}$.

For $n<\omega$, denote
\[B_n=\{a\in \bigcup_{i<n}\msa^i:  \min(a)\leq\alpha_n\leq \max(a)\}\text{ and}\]
\[F_n=cl(\bigcup B_n, \{\msa^i: i<n\}).\]
Then $B_n\subseteq A_n$. 
We check that 
\begin{enumerate}\setcounter{enumi}{1}
\item $\alpha_n \cup F_n$ is $\{\msa^i: i<n\}$-closed.
\end{enumerate}
To see this, recall that $F_n$ is $\{\msa^i: i<n\}$-closed.  Fix $i<n$ and $b\in \msa^i$ with $b\cap (\alpha_n\cup F_n)\neq \emptyset$. If $b\cap F_n\neq\emptyset$, then $b\subseteq F_n$. If $b\cap \alpha_n\neq\emptyset$, then either $b\subseteq \alpha_n$ or $b\in B_n$ and hence $b\subseteq F_n$.

The same argument shows that
\begin{enumerate}\setcounter{enumi}{2}
\item $\alpha_n \cup cl(\bigcup A_n, \{\msa^i: i\leq n\})$ is $\{\msa^i: i\leq n\}$-closed.
\end{enumerate}
By (1), $(\bigcup B_{n+1})\cap (\alpha_n\cup cl(A_n, \{\msa^i: i\leq n\}))=\emptyset$. By (3) and Lemma \ref{lem closure} (ii),
\begin{enumerate}\setcounter{enumi}{3}
\item $F_{n+1}\cap (\alpha_n\cup cl(\bigcup A_n, \{\msa^i: i\leq n\}))=\emptyset$.
\end{enumerate}

We check that $\langle \alpha_n: n<\omega\rangle$  and $\langle F_n\in [\alpha]^{<\omega}: n<\omega\rangle$ are as desired. 

To see (i), note by (1) and the fact $B_n\subseteq A_n$, $F_n\subseteq \alpha_{n+1}$.
Together with (4), $F_n\subseteq [\alpha_{n-1}, \alpha_{n+1})$ and $F_n$'s are pairwise disjoint.

To see (ii), we only need to verify that $\alpha_{n-1}\notin F_n$. Suppose otherwise. Note $\bigcup B_n$ and hence $F_n$ is a union of elements in $\bigcup_{i<n} \msa^i$. 
So for some $i<n$ and $a\in \msa^i$, $\alpha_{n-1}\in a$. But then $a\in A_{n-1}$ and $\alpha_{n-1}\in \bigcup A_{n-1}$. This contradicts (4).

To see (iii), fix $n$ and $\beta\in [\alpha_n, \alpha_{n+1})$. By   the definition of $F_n$ and (2),  $F_n$,  $\alpha_n\cup F_n$ and $\alpha_{n+1}\cup F_{n+1}$ are all  $\{\msa^i: i<n\}$-closed. By Lemma \ref{lem closure} (ii), 
\[\alpha_{n+1}\cup F_{n+1}\setminus (\alpha_n\setminus F_n)=[\alpha_n, \alpha_{n+1})\cup F_n\cup F_{n+1}\]
 is $\{\msa^i: i<n\}$-closed. So $cl(\{\beta\}, \{\msa^i: i<n\})\subseteq [\alpha_n, \alpha_{n+1})\cup F_n\cup F_{n+1}$.
\end{proof}

We introduce a strengthening of the thin Kurepa family which may be useful in constructing specific structures. The strengthening is ccc destructible.
\begin{defn}
A thin Kurepa family $\mcf$ is  \emph{cofinal} if the following statement holds.
\begin{enumerate}[{\rm (K2)$^+$}]
\item Every uncountable pairwise disjoint family $\msa\subseteq [\omega_1]^{<\omega}$ has an uncountable subfamily in $\mcf$.
\end{enumerate}
\end{defn}

Recall \cite[Theorem III.7.6]{Kunen}, $\diamondsuit^+$ implies the existence of a cofinal Kurepa family, i.e., a Kurepa family $\mcf\subseteq P(\omega_1)$ such that every $A\in [\omega_1]^{\omega_1}$ has an uncountable subset in $\mcf$.  We apply this idea to induce a cofinal thin Kurepa family  from a thin Kurepa family under $\diamondsuit^+$.
\begin{lem}\label{lem ctkf}
Assume $\diamondsuit^+$. If there is a thin Kurepa family, then there is a cofinal one.
\end{lem} 
\begin{proof}
Let  $\mcf^*$ be a thin Kurepa family. Let $\langle\mc{A}_\alpha: \alpha<\omega_1\rangle$ be a $\diamondsuit^+$ sequence in the sense that
\begin{enumerate}
\item $\forall \msa\subseteq [\omega_1]^{<\omega}~\exists C~ (C \text{ is a club } \wedge \forall \alpha\in C~ (\msa\cap [\alpha]^{<\omega}\in \mc{A}_\alpha\wedge C\cap \alpha\in \mc{A}_\alpha))$.
\end{enumerate}
Assume in addition that
\begin{enumerate}\setcounter{enumi}{1}
\item $\emptyset\in \mc{A}_0$ and $\mc{A}_\alpha\subseteq \mc{A}_\beta$ for $\alpha<\beta$.
\end{enumerate}

For $\alpha<\beta<\omega_1$ and a pairwise disjoint family $\msa$ with $\msa\cap [\beta]^{<\omega}\setminus [\alpha]^{<\omega}\neq \emptyset$, define $f(\msa, \alpha, \beta)$ to be the element  $a\in\msa\cap [\beta]^{<\omega}\setminus [\alpha]^{<\omega}$ with $\min(a)$ minimal.

For a pairwise disjoint family $\msa$ and  $C\subseteq \omega_1$, $g(\msa, C)$ is the collection of $f(\msa, C(\alpha), C(\alpha+1))$ such that $\alpha<\omega_1$ and $\msa\cap [C(\alpha+1)]^{<\omega}\setminus [C(\alpha)]^{<\omega}\neq \emptyset$.

Now for $\beta<\omega_1$, let $\mcf_\beta$ be the collection of all $a\cup g(\msa\cap \msb, C)$ such that 
\begin{itemize}
\item $a\in [[\beta]^{<\omega}]^{<\omega}$, $\msa\in \mcf^*\up\beta$, $\msb\in \mc{A}_\beta\cap P([\beta]^{<\omega})$ and $C\in \mc{A}_\beta$ is closed in $\beta$;
\item  $a\cup g(\msa\cap \msb, C)$ is non-overlapping.
\end{itemize}

Let $\mcf$ be the collection of all uncountable $g(\msa\cap \msb, C)$ such that $\msa\in \mcf^*$, $\msb\subseteq [\omega_1]^{<\omega}$ and $g(\msa\cap \msb, C)\cap [\beta]^{<\omega}\in \mcf_\beta$ for all $\beta<\omega_1$. We check that $\mcf$ is as desired. First note that every element $\msa\in\mcf$ is an uncountable subset of some $\msb\in \mcf^*$.

(K1) follows from the fact that every $\mcf_\beta$ is countable.

To see (K2)$^+$, fix an uncountable pairwise disjoint family $\msb\subseteq [\omega_1]^{<\omega}$. By (1), find a club $C$  such that for all $\beta\in C$, $\msb\cap [\beta]^{<\omega}\in \mc{A}_\beta$ and $C\cap \beta\in \mc{A}_\beta$. Then by (K2) for $\mcf^*$, find $\msa\in \mcf^*$ such that $\msa\cap \msb$ is uncountable. Now it is straightforward to verify that $g(\msa\cap \msb, C)\in \mcf$. Now (K2)$^+$ follows from the fact that $g(\msa\cap \msb, C)\subseteq \msa\cap \msb\subseteq \msb$.

To see (K3), fix $a\in [\omega_1]^{<\omega}$ and $\mcf'\in [\mcf]^{<\omega}$. Find $\mcf''\in [\mcf^*]^{<\omega}$ such that every element $\msa\in\mcf'$ is contained in some $\msb\in \mcf''$.  Then $cl(a, \mcf')\subseteq cl(a, \mcf'')$ is finite.
\end{proof}

The cofinal property (K2)$^+$ might be useful in constructing specific structures, especially with $\diamondsuit^+$. 
However, unlike the thin Kurepa family, the cofinal property (K2)$^+$ is ccc destructible. For example, adding $\omega_1$ Cohen reals, $Fn(\omega_1, 2)$, destroys (K2)$^+$, even if we take the completion of the thin Kurepa family $\mcf$ in the extension: 
\[\{\msa\subseteq [\omega_1]^{<\omega}: \forall \alpha<\omega_1 ~\msa\cap [\alpha]^{<\omega}\in \mcf\up \alpha\}.\]
Note that (K3) is preserved under taking completion: For $a\in [\omega_1]^{<\omega}$ and finite $\mcf'\subseteq P([\omega_1]^{<\omega})$ consisting of pairwise disjoint families, $|cl(a, \mcf')|\leq \omega$ and hence $cl(a, \mcf')=cl(a, \mcf'\up \alpha)$ for some $\alpha<\omega_1$.

Thin Kurepa families capture uncountable pairwise disjoint families. Note that every uncountable family $\subseteq [\omega_1]^{<\omega}$ contains an uncountable subfamily that is a $\Delta$-system. To capture  arbitrary uncountable families $\subseteq [\omega_1]^{<\omega}$, we only need to add roots to elements in a thin Kurepa family.

\begin{defn}
 For $a\subseteq \omega_1$ and $\mcf\subseteq P([\omega_1]^{<\omega})$ consisting of $\Delta$-systems, $cl(a,\mcf)$ is the $\subseteq$-least subset $a^*$ of $\omega_1$ such that
\begin{itemize}
\item $a\cup \bigcup\{root(\msa): \msa\in \mcf\}\subseteq a^*$ where $root(\msa)$ is the root of $\msa$;
\item   for every $\msa\in \mcf$ and every $b\in \msa$, either $(b\setminus root(\msa))\cap a^*=\emptyset$ or $b\subseteq a^*$.
\end{itemize}
\end{defn}

It turns out that adding roots to a thin Kurepa family preserves the corresponding properties.
\begin{lem}
If there is a thin Kurepa family, then there is a family $\mcf\subseteq P([\omega_1]^{<\omega})$ that consists of $\Delta$-systems and satisfies {\rm (K1)}, {\rm (K3)} and the following statement.
\begin{enumerate}[{\rm (K2)$'$}]
\item For every uncountable  $\msa\subseteq [\omega_1]^{<\omega}$, there exists $\msb\in \mcf$ with  $|\msa\cap\msb|=\omega_1$.
\end{enumerate}
\end{lem}
\begin{proof}
Let $\mcf^*$ be a thin Kurepa family. For $a\in [\omega_1]^{<\omega}$ and $\msa\subseteq [\omega_1]^{<\omega}$, let $a+\msa=\{a\cup b: b\in \msa\}$.  Let
\[\mcf=\{a+\msa: a\in  [\omega_1]^{<\omega}, \msa\in \mcf^*\}.\]
We check that $\mcf$ is as desired. Note for $a\in  [\omega_1]^{<\omega}$ and  $\msa\in \mcf^*$, $a+\msa$ is a $\Delta$-system with root $a$.

(K1) follows from the fact that $\mcf\up\alpha=\{a+\msa: a\in [\alpha]^{<\omega}, \msa\in \mcf^*\up\alpha\}$.

To see (K2)$'$, note that every uncountable family $\subseteq [\omega_1]^{<\omega}$ has an uncountable subfamily that is a $\Delta$-system. Then (K2)$'$ follows from (K2) for $\mcf^*$ and our definition of $\mcf$.

To see (K3), fix $a\in [\omega_1]^{<\omega}$ and $\mcf'\in [\mcf]^{<\omega}$. Find $b\in [\omega_1]^{<\omega}$ and $\mcf''\in [\mcf^*]^{<\omega}$ such that $\mcf'\subseteq \{b'+\msa: b'\subseteq b, \msa\in \mcf''\}$. It is straightforward to check that $cl(a, \mcf')\subseteq cl(a\cup b, \mcf'')$ and hence is finite.
\end{proof}

The corresponding version for the cofinal thin Kurepa family is also true.\bigskip

For the rest of the section, we construct a thin Kurepa family.  We start from a strong Kurepa family induced from a cofinal Kurepa family.
\begin{fact}\label{fact1}
Assume $\diamondsuit^+$. Then there is a family $\mcf\subseteq P([\omega_1]^{<\omega})$ such that $\mcf$ consists of uncountable non-overlapping families and satisfies {\rm (K1)} and {\rm (K2)$^+$}.
\end{fact}
\begin{proof}
By \cite[Theorem III.7.6]{Kunen} (see also the proof of Lemma \ref{lem ctkf}), a bijection between $\omega_1$ and $[\omega_1]^{<\omega}$ induces a family $\mcf'\subseteq P([\omega_1]^{<\omega})$ such that (K1) and (K2)$^+$ hold. Then 
$\mcf=\{\msa\in \mcf': \msa\text{ is an uncountable non-overlapping family}\}$
 is as desired.
\end{proof}

We then force a thin Kurepa family from a family $\mcf\subseteq P([\omega_1]^{<\omega})$ satisfying the conclusion of Fact \ref{fact1}.
\begin{defn}\label{defn pf}
 Suppose $\mcf\subseteq P([\omega_1]^{<\omega})$ is a family satisfying the conclusion of Fact \ref{fact1}. 
 \begin{enumerate}
 \item $\mcf_{succ}=\{\msa\cap [\max(a)+1]^{<\omega}: \msa\in \mcf, a\in \msa\}$ and for $\msa\in \mcf_{succ}$,  $\max(\msa)$ is the $a\in \msa$ such that $\msa\subseteq [\max(a)+1]^{<\omega}$.
 \item $\pp_\mcf$ is the poset consisting of finite partial maps from $\mcf_{succ}$ to 2 ordered by function extensions.
 \end{enumerate}
\end{defn}
Before going into details, we first explain the role that $\pp_\mcf$ will play. The thin Kurepa family will be induced from a generic filter over $\pp_\mcf$.

\begin{defn}\label{defn tkf}
Suppose $\pp_\mcf$ is as in Definition \ref{defn pf} and $G$ is a generic filter over $\pp_\mcf$. For $\msa\in \mcf$,  $G(\msa)=\{a\in \msa: \text{ for some } p\in G, ~ p(\msa\cap [\max(a)+1]^{<\omega})=1\}$.
$G(\mcf)=\{G(\msa): \msa\in \mcf\}$.
\end{defn}

 $G(\mcf)$ will be a thin Kurepa family.  (K1) follows from the fact that $G(\msa)\cap [\alpha]^{<\omega}$ is determined by $\msa\cap [\alpha]^{<\omega}$.  (K2) and (K3) will follow from a density argument.

Note that $|\mcf_{succ}|=\omega_1$ and hence $\pp_\mcf$ is isomorphic to $Fn(\omega_1, 2)$. 
\begin{lem}\label{lem ccc}
$\pp_\mcf$ in Definition \ref{defn pf} is ccc.
\end{lem}

We check (K1)-(K3) for  $G(\mcf)$ in the following three lemmas.
\begin{lem}\label{lem K1}
For every $\alpha<\omega_1$, $|G(\mcf)\up \alpha|\leq \omega$.
\end{lem}
\begin{proof}
Fix $\alpha<\omega_1$. By definition of $\pp_\mcf$ and $G(\mcf)$, for $\msa$ and $\msb$ in $\mcf$, $G(\msa)\cap [\alpha]^{<\omega}=G(\msb)\cap [\alpha]^{<\omega}$ if $\msa\cap [\alpha]^{<\omega}=\msb\cap [\alpha]^{<\omega}$. So $|G(\mcf)\up \alpha|\leq |\mcf\up \alpha|\leq \omega$.
\end{proof}

\begin{lem}\label{lem K3}
For all $a\in [\omega_1]^{<\omega}$ and all $\mcf'\in [\mcf]^{<\omega}$, $cl(a, G(\mcf'))$ is finite where $G(\mcf')=\{G(\msa): \msa\in \mcf'\}$.
\end{lem}
\begin{proof}
Fix $a, \mcf'$ and $p\in \pp_\mcf$. Let 
\[A=a\cup \bigcup\{\max(\msa): \msa\in dom(p)\}\text{ and}\]
\[\mcx=\{b: b\in \msa\text{ for some }\msa\in \mcf', ~ b\not\subseteq A\text{ and } b\cap A\neq \emptyset\}.\]
Since $\mcf'$ is finite and consists of pairwise disjoint families, $\mcx$ is finite.

Let $p'=\{(\msa\cap [\max(b)+1]^{<\omega}, 0): \msa \in \mcf', b\in \mcx\cap \msa\}$. By definition of $\mcx$, $dom(p')\cap dom(p)=\emptyset$. Then $q=p\cup p'$ extends $ p$.

We now check that $q\Vdash cl(a, \dot{G}(\mcf'))\subseteq A$. Fix $q'\leq q$, $\msa\in \mcf'$ and $b\in \msa$ with 
\[q'\Vdash b\in  \dot{G}(\msa).\]
Equivalently, 
\begin{enumerate}
\item  $q'(\msa\cap [\max(b)+1]^{<\omega})=1$.
\end{enumerate}
  It suffices to prove that either $b\subseteq A$ or $b\cap A=\emptyset$.

Suppose $b\cap A\neq \emptyset$. 
   By definition of $p'$, (1) and the fact $q'\leq q\leq p'$, $b\notin \mcx$. Together with $b\cap A\neq\emptyset$, $b\subseteq A$. 

This shows that $q\Vdash cl(a, \dot{G}(\mcf'))\subseteq A$ and hence is finite. Now the lemma follows from a density argument.
\end{proof}

\begin{lem}\label{lem K2}
$G(\mcf)$ defined in Definition \ref{defn tkf} is a thin Kurepa family.
\end{lem}
\begin{proof}
By Lemma \ref{lem K1}-\ref{lem K3}, we only need to verify that $G(\mcf)$ satisfies (K2). Fix $p\in \pp_\mcf$ and $\dot{\msa}$ such that
\[p\Vdash \dot{\msa}\subseteq [\omega_1]^{<\omega}\text{ is an uncountable pairwise disjoint family}.\]

For every $\alpha<\omega_1$, find $p_\alpha\leq p$ and $a_\alpha\in [\omega_1\setminus \alpha]^{<\omega}$ such that
\[p_\alpha\Vdash a_\alpha\in \dot{\msa}.\]

Find a stationary $\Gamma\subseteq \omega_1$ and $\overline{\alpha}<\omega_1$ such that
\begin{enumerate}
\item  for every $\alpha\in \Gamma$ and every $\msa\in dom(p_\alpha)$, either $\msa\cap [\overline{\alpha}]^{<\omega}\neq \emptyset$ or $\msa\cap [\alpha]^{<\omega}=\emptyset$.
\end{enumerate}
Shrinking $\Gamma$, we may assume that $\min (\Gamma)>\overline{\alpha}$ and  $\{a_\alpha: \alpha\in \Gamma\}$ is pairwise disjoint. 

By Fact \ref{fact1}, find $\Gamma'\in [\Gamma]^{\omega_1}$ such that 
\[\msb=\{a_\alpha: \alpha\in \Gamma'\}\in \mcf.\]
 Note for $\alpha\in \Gamma'$, $\min(a_\alpha)\geq \alpha\geq\min(\Gamma)>\overline{\alpha}$. So $\msb\cap [\overline{\alpha}]^{<\omega}=\emptyset$. 

Let $\overline{\beta}=\max(a_{\min(\Gamma')})+1$. Then $a_{\min(\Gamma')}\in \msb\cap [\overline{\beta}]^{<\omega}$.
By (1), 
\begin{enumerate}\setcounter{enumi}{1}
\item for every $\alpha\in \Gamma'\setminus \overline{\beta}$ and every $\msa\in  dom(p_\alpha)$, $\msb\cap [\overline{\beta}]^{<\omega}\neq\msa \cap [\overline{\beta}]^{<\omega}$.
\end{enumerate}

Now for $\alpha\in \Gamma'\setminus \overline{\beta}$, let
\[q_\alpha=p_\alpha\cup \{(\msb\cap [\max(a_\alpha)+1]^{<\omega}, 1)\}.\]
By (2) and the fact $\min(a_\alpha)\geq \alpha\geq \overline{\beta}$, it is straightforward to check that $q_\alpha$ is a map and hence $q_\alpha\leq p_\alpha$. Moreover, $q_\alpha\Vdash a_\alpha\in \dot{\msa}\cap \dot{G}(\msb)$.

By Lemma \ref{lem ccc}, there exists $q\leq p$ such that
\[q\Vdash \{\alpha\in \Gamma'\setminus \overline{\beta}: q_\alpha\in \dot{G}\}\text{ is uncountable}\]
where $\dot{G}$ is the canonical name of the generic filter. Consequently, 
\[q\Vdash \dot{\msa}\cap \dot{G}(\msb)\supseteq \{a_\alpha:  \alpha\in \Gamma'\setminus\overline{\beta}, q_\alpha\in \dot{G}\}\text{ is uncountable}.\]
Now (K2) follows from a density argument.
\end{proof}

\begin{thm}\label{thm consistency}
It is relative consistent with {\rm ZFC} that there is a thin Kurepa family and {\rm GCH}$+\diamondsuit^+$ hold.
\end{thm}
\begin{proof}
Start from a model $V$ of GCH$+\diamondsuit^+$. Fix a family $\mcf\subseteq P([\omega_1]^{<\omega})$ satisfying the conclusion of Fact \ref{fact1}. Then force with $\pp_\mcf$ as   in Definition \ref{defn pf}. Let $G$ be a generic filter over $\pp_\mcf$ and $G(\mcf)$ be as in Definition \ref{defn tkf}.

By Lemma \ref{lem K2}, $G(\mcf)$ is a thin Kurepa family. By Lemma \ref{lem ccc}, $\pp_\mcf$ is ccc and has size $\omega_1$. So GCH$+\diamondsuit^+$ remain true in $V[G]$.
\end{proof}
The following question is natural.
\begin{question}
Does $\diamondsuit^+$ imply the existence of a thin Kurepa family?
\end{question}
Additional structure might be associated with a thin Kurepa family. In the next section, we will transform a thin Kurepa family to an appropriate family on a complete coherent Suslin tree $S$ and then construct a coloring on $[S]^2$ with some precaliber $\omega_1$ type property (see Theorem \ref{thm precaliber} below).

We would like to point out that the thin Kurepa family is quite flexible and may be used to construct other types of families, e.g., the damage control structure introduced in \cite[Section 7]{Peng}. Assuming $2^{\omega_1}=\omega_2$+the existence of a cofinal thin Kurepa family,  there exists $(\mc{C}, \mc{T}, \mathbf{E})$ such that $\varphi_1(\mc{C}, \mc{T}, \mathbf{E})$ holds (see \cite[Definition 12]{Peng}) and $\mc{C}\subseteq P([\omega_1]^{<\omega})$ satisfies (K1), (K3) and the following statement.
\begin{itemize}
\item For every uncountable $\msa\subseteq [\omega_1]^{<\omega}$, there are $\msb\in \mc{C}$ and $I\in [\omega]^{<\omega}$ such that $\{b[I]: b\in \msb\}\subseteq \msa$ where $b[I]=\{b(i): i\in I\}$.
\end{itemize}
One may also define a coloring on $[\omega_1]^2$ (or $[\omega_1]^{<\omega}$) according to the structure $(\mc{C}, \mc{T}, \mathbf{E})$ as in \cite{Peng}. Of course, this direct construction cannot guarantee the other key property, the corresponding fragment of Martin's axiom, which is    iteratively forced in \cite[Section 8-9]{Peng}.

\section{A strong coloring on a Suslin tree}

Throughout this section, we fix a complete coherent Suslin tree $S\subseteq 2^{<\omega_1}$ and assume the existence of a thin Kurepa family. To simplify our notation, we will use the following linear order on $S$.

  \begin{defn}\label{ho}
  \begin{enumerate}
\item   The \emph{height order}, denoted by $<_{ho}$, is the linear order on  $2^{<\omega_1}$ defined by letting for any $s\neq t$ in  $2^{<\omega_1}$,  $s<_{ho} t$ if any of the following conditions holds:
\begin{itemize}
\item $ht(s)<ht(t)$;
\item $ht(s)=ht(t)$ and $ s(\Delta(s,t))<t(\Delta(s,t))$.
\end{itemize}
  \item For $a\in [2^{<\omega_1}]^{<\omega}$, $a(0),a(1),...,a(|a|-1)$ is the $\lo$-increasing enumeration of $a$. 
  \end{enumerate}
   \end{defn}

We will define a strong coloring $c$ on $S$. Before defining $c$, we first transform the thin Kurepa family to appropriate subfamilies of $[S]^{<\omega}$.

\begin{defn}\label{def2}
For    $s\in S$ and $0<n<\omega$, $\mcx_{s,n}$ is the collection of all $a\in [S]^{2n}$ such that
\begin{enumerate}
\item $s^\smallfrown 0<_S a(0)<_S...<_S a(n-1)$ and $s^\smallfrown 1<_S a(n)<_S...<_S a(2n-1)$;
\item $ht(a(n-1))< ht(a(n))$ and $D_{a(n-1), a(n)}=\{ht(s)\}$.\footnote{Recall that $D_{x, y}=\{\alpha<\min\{ht(x), ht(y)\}: x(\alpha)\neq y(\alpha)\}$.}
\end{enumerate}
$\mcx=\bigcup \{\mcx_{s,n}: s\in S, 0<n<\omega\}$.
\end{defn}

A bijection between $\omega_1$ and $S$ induces a  Kurepa family $\subseteq P([S]^{<\omega})$ from a thin Kurepa family $\subseteq P([\omega_1]^{<\omega})$. We will further transform the  Kurepa family $\subseteq P([S]^{<\omega})$ to ``guess''  uncountable pairwise disjoint families $\msa\subseteq \mcx_{s,n}$ where   $s\in S$ and $0<n<\omega$.

Say $\msa\subseteq [S]^{<\omega}$ is \emph{non-overlapping} if for $a\neq b$ in $\msa$, either $\max\{ht(s): s\in a\}<\min\{ht(t): t\in b\}$ or $\max\{ht(t): t\in b\}<\min\{ht(s): s\in a\}$.

\begin{lem}\label{lem stkf}
Suppose there is a thin Kurepa family. Then there is a family $\mcf\subseteq \bigcup\{[\mcx_{s,n}]^{\omega_1}: s\in S, 0<n<\omega\}$ with the following properties.
\begin{enumerate}[(i)]
\item For every $\alpha<\omega_1$, $|\{\msa\cap [S\up\alpha]^{<\omega}: \msa\in \mcf\}|\leq \omega$.
\item For every $s\in S$, every $0<n<\omega$ and every uncountable pairwise disjoint family $\msa\subseteq \mcx_{s, n}$, there exists $\msb\in \mcf$ with $|\msa\cap \msb|=\omega_1$.
\item For every $a\in [S]^{<\omega}$ and every $\mcf'\in [\mcf]^{<\omega}$, $cl(a, \mcf')$ is finite.
\item Every $\msa\in \mcf$ is non-overlapping.
\end{enumerate}
\end{lem}
\begin{proof}
Fix a thin Kurepa family $\mcf^*\subseteq P([\omega_1]^{<\omega})$.  Fix a bijection $\pi: \omega_1\ra S$ such that 
\[ht(\pi(\alpha))\leq ht(\pi(\beta))\text{ whenever }\alpha<\beta.\] 
Denote $\pi(\msa)=\{\pi[a]: a\in \msa\}$ for $\msa\subseteq [\omega_1]^{<\omega}$.
Let 
\[\mcf'=\{\pi(\msa)\cap \mcx_{s,n}: \msa\in \mcf^*, s\in S, 0<n<\omega, |\pi(\msa)\cap \mcx_{s,n}|=\omega_1\}.\]

Define a partial order $<_w$ on $[S]^{<\omega}$ by 
\[a<_w b \text{ iff } \max(\pi^{-1}[a])<\min(\pi^{-1}[b]).\]
Recall that every $\msa\in \mcf^*$ is non-overlapping. So $<_w$ is a well-ordering on $\pi(\msa)$. For $\msb\in \mcf'$, let $even(\msb)$ be the collection of elements of $\msb$ that occupy even positions  in $\msb$ under $<_w$ and $odd(\msb)=\msb\setminus even(\msb)$.
 We show that
 \[\mcf=\{even(\msb), odd(\msb): \msb\in \mcf'\}\]
 is as desired.
 
 To see (i), note for $s\notin S\up\alpha$, $\mcx_{s,n}\cap [S\up\alpha]^{<\omega}=\emptyset$. So
 \begin{align*}
 &\{\msa\cap [S\up\alpha]^{<\omega}: \msa\in \mcf'\}\\
 \subseteq &\{\emptyset\}\cup \{\pi(\msa)\cap \mcx_{s,n}\cap [S\up\alpha]^{<\omega}: \msa\in \mcf^*, s\in S\up\alpha, 0<n<\omega\}
 \end{align*}
  is countable. 
 Then 
 \[\{\msa\cap [S\up\alpha]^{<\omega}: \msa\in \mcf\}=\{even(\msa), odd(\msa): \msa\in\{\msb\cap [S\up\alpha]^{<\omega}: \msb\in \mcf'\}\}\]
  is countable.

 (ii) follows from the definition of $\mcf$ and the fact that $\mcf^*$ is a thin Kurepa family.  
 
 To see (iii), fix  $a\in [S]^{<\omega}$ and $\mcf''\in [\mcf]^{<\omega}$. Find $\mcf'''\in [\mcf^*]^{<\omega}$ such that every $\msa\in \mcf''$ is contained in $\pi(\msb)$ for some $\msb\in \mcf'''$. Then $cl(a, \mcf'')\subseteq cl(a, \{\pi(\msb): \msb\in \mcf'''\})$ and hence is finite.
 
 To see (iv), fix $\msb=even(\pi(\msa)\cap \mcx_{s,n})$ in $\mcf$ and the argument works for $odd(\pi(\msa)\cap \mcx_{s,n})$ as well. Fix
 $a\neq b$ in $\msb$. Recall that $\msa$ is non-overlapping. By symmetry, assume $a<_w b$. Then there exists $a'\in odd(\pi(\msa)\cap \mcx_{s,n})$ such that $a<_w a'<_w b$.
 By our choice of $\pi$, $ht(a(2n-1))\leq ht(a'(0))<ht(a'(2n-1))\leq ht(b(0))$.
\end{proof}

\begin{defn}
Say $\mcf$ is a \emph{2-thin Kurepa family wrt $S$} if $\mcf$ satisfies the conclusion of Lemma \ref{lem stkf}.
\end{defn}
We would like to point out that above definition can be extended to all complete coherent $\omega_1$-tree $S\subseteq 2^{<\omega_1}$.

The proof of Lemma \ref{lem tkf ladder} induces a corresponding partition of $\alpha$ into $\omega$ intervals for countable limit $\alpha$.

\begin{lem}\label{lem 2tkf ladder}
Suppose $\mcf$ is a 2-thin Kurepa family wrt $S$ and $\{\msa\cap [S\up\alpha]^{<\omega}: \msa\in \mcf\}=\{\msa^n: n<\omega\}$ where $\alpha$ is a countable limit ordinal. Then there are an increasing sequence of ordinals $\langle \alpha_n: n<\omega\rangle$ cofinal in $\alpha$ and a sequence $\langle F_n\in [S\up\alpha]^{<\omega}: n<\omega\rangle$  satisfying the following properties.
\begin{enumerate}[(i)]
\item $F_n$'s are pairwise disjoint.
\item For every $n<\omega$, $F_n=cl(F_n, \{\msa^i: i<n\}) \subseteq \bigcup_{\alpha_{n-1}<\beta< \alpha_{n+1}} S_\beta$ and $F_n\cup {S\up \alpha_n}=cl(F_n\cup S\up \alpha_n, \{\msa^i: i<n\})$.
\item For every $n<\omega$ and every $s\in\bigcup_{\alpha_n\leq\beta< \alpha_{n+1}} S_\beta$, 
\[cl(\{s\}, \{\msa^i: i<n\})\subseteq \bigcup_{\alpha_n\leq\beta< \alpha_{n+1}} S_\beta \cup F_n\cup F_{n+1}\]
\end{enumerate}
\end{lem}

\begin{thm}\label{thm coloring}
Suppose   $S\subseteq 2^{<\omega_1}$ is a complete coherent Suslin tree and $\mcf$ is a 2-thin Kurepa family wrt $S$. Then there is a coloring $c: [S]^2\ra 2$ satisfying the following properties.
\begin{enumerate}
\item[(coh)] For $s<_S t$ in $S$, $\{x<_S s: c(x, s)\neq c(x, t)\}$ is finite.
\item[(hom)] For every $\alpha<\omega_1$,  $s\in S\up\alpha$, $0<n<\omega$, $\msa\in \mcf\cap [\mcx_{s, n}]^{\omega_1}$, and every $a\in [S_\alpha]^2$ with $D_{a(0), a(1)}=\{ht(s)\}$,  for all but finitely many $b\in \msa\cap [S\up \alpha]^{<\omega}$, there exists $i<2$ such that if $b(2n-1)<_S a(1)$, then
\[c(b(j), a(0))=i\text{ and }c(b(n+j), a(1))=1-i\text{ for all }j<n.\]
\end{enumerate}
\end{thm}
\begin{proof}
We first define $c(\{s, t\})=1$ if $s$ is incomparable with $t$. For $s<_S t$, use $c(s, t)$ to denote $c(\{s, t\})$. We will then define $c_s: S_{<s}\ra 2$ for all $s\in S$ and let $c(x, s)=c_s(x)$. We define $\{c_s: s\in S_\alpha\}$ by induction on $\alpha$.

For $s\in S, 0<n<\omega$, $a\in \mcx_{s,n}$ and $u, v\in S$ with $D_{u,v}=\{ht(s)\}$ and $a(n-1)<_S u$, $a(2n-1)<_S v$, say $(a, u, v)$ is $c$-good if there exists $i<2$ such that $c_u(a(j))=i$ and $c_v(a(n+j))=1-i$ for all $j<n$.

Note that (coh) and (hom) are preserved by finite modifications. So the successor steps are not important. Define $c_s$, for $s\in S_{\alpha+1}$, by $c_s=c_{s\up \alpha}\cup \{(s\up\alpha,0)\}$.

Now suppose $\alpha$ is a limit ordinal and $\{c_s: s\in S\up_\alpha\}$ is defined such that (coh) and (hom) are satisfied up to $\alpha$. Let $\{s_n: n<\omega\}=S_\alpha$  and $\{\msa^n: n<\omega\}=\{\msa\cap [S\up\alpha]^{<\omega}: \msa\in \mcf\}$.

Choose sequences $\langle \alpha_n: n<\omega\rangle$ and $\langle F_n: n<\omega\rangle$ satisfying the conclusion of Lemma \ref{lem 2tkf ladder}.

For $i<n<\omega$, denote
\[D_{i,n}=S_{<s_i\up \alpha_n}\setminus F_n.\]
We now define $\langle c_{s_i}\up D_{i,n}: i<n\rangle$ by induction on $n$. Note that
\[S_{<s_i\up \alpha_{n-1}}\subseteq D_{i,n}\subseteq S_{<s_i\up \alpha_n}.\]

There is nothing to define for $n=0$ and so we describe the definition for $n=k+1$.

Let $E_{k}$ be the collection of $a\in (\bigcup_{i< k} \msa^i)\cap [\bigcup_{\alpha_k\leq \xi <\alpha_{k+1}} S_\xi]^{<\omega}$ such that
\begin{itemize}
\item $a\cap F_{k+1}=\emptyset$ and for some $ i, j< k$,
$|D_{s_i, s_j}|=1$, $ a(\frac{|a|}{2}-1)<_S s_i$, $a(|a|-1)<_S s_j$ and $(a, s_i\up \alpha_{k+1}, s_j\up \alpha_{k+1})$ is not $c$-good.
\end{itemize}
Note by induction hypothesis at step $\alpha_{k+1}$, $E_{k}$ is finite. Let
\begin{align*}
X_{k}= F_k\cup cl(\bigcup E_{k}, \{\msa^i: i< k\}).
\end{align*}

Note $\bigcup E_{k}\subseteq \bigcup_{\alpha_k\leq \xi <\alpha_{k+1}} S_\xi\setminus F_{k+1}$. By Lemma \ref{lem 2tkf ladder}, 
\[cl(\bigcup E_{k}, \{\msa^i: i< k\})\subseteq \bigcup_{\alpha_k\leq \xi <\alpha_{k+1}} S_\xi \cup F_k\setminus F_{k+1}.\]
Since $F_k$ is disjoint from $F_{k+1}$,
\begin{enumerate}
\item  $F_k\subseteq X_{k}\subseteq \bigcup_{\alpha_k\leq \xi <\alpha_{k+1}} S_\xi \cup F_k\setminus F_{k+1}$.
\end{enumerate}

Fix $i< k$. 
  Note by (1),
  \[X_{k}\cap S_{<s_i\up\alpha_k}=S_{<s_i\up\alpha_k}\setminus D_{i, k}\]
   and hence $c_{s_i}$ is undefined on $X_{k}\cap S_{<s_i}$ at step $k$. Let $l_i<2$ be such that 
\[|D_{s_0, s_i}|\equiv l_i\text{ mod 2}.\]
First
\begin{enumerate}\setcounter{enumi}{1}
\item define $c_{s_i}$ to be constant $l_i$ on $X_{k}\cap S_{<s_i}$.
\end{enumerate}
  Then 
  \begin{enumerate}\setcounter{enumi}{2}
\item define $c_{s_i}$ to equal $c_{s_i\up \alpha_{k+1}}$ on 
\[\{x\in S_{<s_i}: \alpha_k\leq ht(x)<\alpha_{k+1}\}\setminus (X_{k}\cup F_{k+1}).\]
\end{enumerate}
By (1)-(3), $c_{s_i}$ is defined  on 
\[D_{i,k}\cup (X_{k}\cap S_{<s_i})\cup \{x\in S_{<s_i}: \alpha_k\leq ht(x)<\alpha_{k+1}\}\setminus  F_{k+1}=D_{i, k+1}.\]

Finally 
\begin{enumerate}\setcounter{enumi}{3}
\item define $c_{s_k}$ to equal $c_{s_k\up \alpha_{k+1}}$ on $D_{k,k+1}$.
\end{enumerate}

This finishes the induction at step $n=k+1$. 

Now suppose $\langle c_{s_i}\up D_{i,n}: i<n\rangle$ is defined for all $n$. Then for every $i<\omega$, $c_{s_i}$ is defined on $\bigcup_{i<n<\omega} D_{i,n}=S_{<s_i}$.

We check that (coh) and (hom) are satisfied for $\{c_s: s\in S\up \alpha+1\}$.

To show (coh), it suffices to prove that for every $i<n<\omega$, $c_{s_i}$ agrees with $c_{s_i\up \alpha_n}$ for all but finitely points in $S_{<s_i\up\alpha_n}$. Fix $i<\omega$ and we prove by induction on $n$. 

For $n=i+1$, note  by (4),
\[\{x<_S s_i\up \alpha_{i+1}: c_{s_i}(x)\neq c_{s_i\up \alpha_{i+1}}(x)\}\subseteq S_{<s_i\up\alpha_{i+1}}\setminus D_{i,i+1}\subseteq F_{i+1}\]
 and is finite.
 
 Now suppose that the statement is true for $n=k>i$ and we prove for $n=k+1$. By (3),
 \[\{x\in S_{<s_i\up\alpha_{k+1}}:  ht(x)\geq\alpha_k, c_{s_i}(x)\neq c_{s_i\up\alpha_{k+1}}(x)\}\subseteq X_{k}\cup F_{k+1}\]
 and is finite. Then by induction hypothesis on $n=k$ and induction hypothesis of (coh) at step $\alpha_{k+1}$, both $c_{s_i\up\alpha_{k+1}}$ and $c_{s_i}$ agree with $c_{s_i\up\alpha_k}$ on all but finitely many points of $S_{<s_i\up\alpha_k}$. So $c_{s_i}$ agrees with $c_{s_i\up\alpha_{k+1}}$ on all but finitely many points of $S_{<s_i\up \alpha_{k+1}}$.
 
 This shows (coh) and we check (hom). Fix $s\in S\up \alpha$, $0<n'<\omega$, $\msa\in \mcf\cap [\mcx_{s,n'}]^{\omega_1}$ and $i, j<\omega$ with $D_{s_i, s_j}=\{ht(s)\}$ and $s_i(ht(s))=0$. Assume that $\msa\cap [S\up\alpha]^{<\omega}=\msa^n$. 
 
 It is equivalent to prove that $(a, s_i, s_j)$ is $c$-good for all but finitely many $a\in \msa^n$  with $a(2n'-1)<_S s_j$.
 Denote
 \[N=\max\{i, j, n\}+1.\]
 
 First note by induction hypothesis at step $\alpha_N$, $(a, s_i\up\alpha_N, s_j\up\alpha_N)$ is $c$-good for all but finitely many $a\in \msa\cap [S\up\alpha_N]^{<\omega}$  with $a(2n'-1)<_S s_j$. Then by (coh), $(a, s_i, s_j)$ is $c$-good for all but finitely many $a\in \msa\cap [S\up\alpha_N]^{<\omega}$  with $a(2n'-1)<_S s_j$.
 
 Now it suffices to prove that $(a, s_i, s_j)$ is $c$-good for all $a\in \msa^n\setminus [S\up\alpha_N]^{<\omega}$ with $a(2n'-1)<_S s_j$. Fix $a\in \msa^n\setminus [S\up\alpha_N]^{<\omega}$ with $a(2n'-1)<_S s_j$ and we prove by cases.  Assume $\max(a)\in [\alpha_k, \alpha_{k+1})$ for some $k$. Clearly $k\geq N$.
 
  Note by Lemma \ref{lem 2tkf ladder} and the fact $k\geq N>n$, 
   \begin{enumerate}\setcounter{enumi}{4}
 \item $a\subseteq cl(\{\max(a)\},\{\msa^i: i<k\})\subseteq\bigcup_{\alpha_k\leq\beta< \alpha_{k+1}} S_\beta \cup F_k\cup F_{k+1}$.
  \end{enumerate}
 \medskip
 
 \textbf{Case 1.} $a\cap (F_k\cup F_{k+1})\neq\emptyset$.\medskip
 
We prove for the case $a\cap F_k\neq\emptyset$ and the argument works for $a\cap F_{k+1}\neq\emptyset$ as well. Since $k\geq N>n$, $a\subseteq F_k$. Then by (1), $a\subseteq X_{k}$. 

Recall that  $l_m<2$ satisfies $|D_{s_0, s_m}|\equiv l_m$ mod 2. 

By definition of $c_{s_i}$  and $c_{s_j}$ on $X_{k}$, $c_{s_i}$ is constant $l_i$ on $\{a(m): m<n'\}$ and $c_{s_j}$ is constant $l_j$ on $\{a(n'+m): m<n'\}$.
 
 Since $D_{s_i, s_j}=\{ht(s)\}$, $D_{s_0, s_i}\Delta ~D_{s_0, s_j}=\{ht(s)\}$ and hence $l_j=1-l_i$. This shows that $(a, s_i, s_j)$ is $c$-good.\medskip
 
  \textbf{Case 2.} $a\cap X_{k}\neq\emptyset$.\medskip
  
  We may assume that Case 1 does not occur. So $a\cap cl(\bigcup E_{k}, \{\msa^i: i< k\})\neq \emptyset$. Since $k>n$, $a\subseteq cl(\bigcup E_{k}, \{\msa^i: i< k\})$. Then $a\subseteq X_{k}$ and the argument of Case 1 shows that $(a, s_i, s_j)$ is $c$-good.\medskip
  
  \textbf{Case 3.} Otherwise.\medskip
  
  By (5), $a\in [\bigcup_{\alpha_k\leq \xi <\alpha_{k+1}} S_\xi\setminus (F_k\cup F_{k+1})]^{<\omega}$. By definition of $X_{k}$, $a\notin E_{k}$ and $(a, s_i\up \alpha_{k+1}, s_j\up \alpha_{k+1})$ is $c$-good. Since $a\cap (X_{k}\cup F_{k+1})=\emptyset$, $c_{s_i}$ ($c_{s_j}$) is defined according to $c_{s_i\up \alpha_{k+1}}$ ($c_{s_j\up\alpha_{k+1}}$) on $a$. Hence, $(a, s_i, s_j)$ is $c$-good.\medskip
  
  So in any case, $(a, s_i, s_j)$ is $c$-good. This shows (hom) at step $\alpha$. Now the theorem follows from a standard induction.
\end{proof}

It is clear that (coh) and (hom) are absolute. The proof of Lemma \ref{lem ccc indestructible} shows that 2-thin Kurepa family wrt $S$ is ccc indestructible (even if Suslin of $S$ is destroyed).
\begin{lem}\label{lem S ccc indestructible}
Suppose $S\subseteq 2^{<\omega_1}$ is an $\omega_1$-tree, $\mcf$ is a 2-thin Kurepa family wrt $S$ and $\pp$ is a ccc poset. Then 
\[\Vdash_\pp \check{\mcf}\text{ is a 2-thin Kurepa family wrt } S.\]
\end{lem}

\begin{thm}\label{thm precaliber}
Suppose $S\subseteq 2^{<\omega_1}$ is a complete coherent Suslin tree, $\mcf$ is a 2-thin Kurepa family wrt $S$ and $c: [S]^2\ra 2$ is a coloring satisfying (coh) and (hom). Then for every $s\in S$, $0<n<\omega$ and every uncountable  family $\mc{A}\subseteq[\mcx_{s,n}]^{<\omega}$ with $\{\bigcup x: x\in \mc{A}\}$ pairwise disjoint, there exists $\mc{B}\in [\mc{A}]^{\omega_1}$ such that for $x\neq y$ in $\mc{B}$ and $a\in x$, $b\in y$, if $a(2n-1)<_S b(2n-1)$, then for some $i<2$,
\[c(a(j), b(k))=i \text{ and } c(a(n+j), b(n+k))=1-i \text{ whenever } j,k<n.\]
\end{thm}
\begin{proof}
Replacing $\mc{A}$ by an uncountable subset, we may assume that $\{\bigcup x: x\in \mc{A}\}$ is non-overlapping.  Moreover, assume for some $m<\omega$, $|x|=m$ for all $x\in \mc{A}$. 

Let $\{x_\alpha: \alpha<\omega_1\}$ list $\mc{A}$. For $\alpha<\omega_1$, enumerate $x_\alpha$ as $a_{\alpha,0},..., a_{\alpha, m-1}$. By  Lemma \ref{lem stkf} (ii), inductively find $\Gamma_0\supseteq\cdots\Gamma_{m-1}$ in $[\omega_1]^{\omega_1}$ and $\langle \msa^i\in \mcf: i<m\rangle$ such that for every $i<m$,
\begin{enumerate}
\item $\{a_{\alpha, i}: \alpha\in \Gamma_i\}\subseteq \msa^i$.
\end{enumerate}

Let $c_s$ be the map defined on $S_{<s}$ by $c_s(t)=c(t, s)$.

For  every $a\in \mcx_{s,n}$, by (coh),  there is a finite set $F_a\subseteq S$ such that for every $k<n$, $c_{a(k)}$ agrees with $c_{a(0)}$ on $S_{<a(0)}\setminus F_a$ and $c_{a(n+k)}$ agrees with $c_{a(n)\up ht(a(0))}$ on $S_{<a(n)\up ht(a(0))}\setminus F_a$.

For every $a\in \mcx_{s,n}$, by (hom),  there is a finite set $F'_a\subseteq S$ such that for every $b\in \bigcup_{i<m} (\msa^i\cap [S\up ht(a(0))]^{<\omega})$, if $b(2n-1)<_S a(2n-1)$ and $b\cap F'_a=\emptyset$, then for some $i<2$,
\[c(b(j), a(0))=i\text{ and }c(b(n+j), a(n)\up ht(a(0)))=1-i\text{ for all }j<n.\]

Let $F_\alpha=\bigcup_{a\in x_\alpha} F_a\cup F'_a$ for $\alpha\in \Gamma_{m-1}$. 

Then for $\alpha\in\Gamma_{m-1}$, $a\in x_\alpha$ and $b\in \bigcup_{i<m} (\msa^i\cap [S\up ht(a(0))]^{<\omega})$, 
\begin{enumerate}\setcounter{enumi}{1}
\item if $b(2n-1)<_S a(2n-1)$ and $b\cap F_\alpha=\emptyset$, then for some $i<2$,
\[c(b(j), a(k))=i\text{ and }c(b(n+j), a(n+k))=1-i\text{ whenever }j, k<n.\]
\end{enumerate}

By (1), (2) holds for all $\xi\neq\alpha$ in $\Gamma_{m-1}$ and all $b\in x_\xi$, $a\in x_\alpha$.

Recall that $F_\alpha$'s are finite sets and $\{\bigcup x: x\in \mc{A}\}$ is pairwise disjoint. Now find $\Gamma\in [\Gamma_{m-1}]^{\omega_1}$ such that for $\xi\neq\alpha$ in $\Gamma$, $(\bigcup x_\xi)\cap F_\alpha=\emptyset$. Then $\mc{B}=\{x_\alpha: \alpha\in \Gamma\}$ is as desired.
\end{proof}

We recall a property from \cite{PW24}. 
\begin{defn}[\cite{PW24}]
\begin{enumerate}
\item For $n>0$, $\langle (t_i, N_i)\in S\times \omega\setminus\{0\}: i<n\rangle$ and non-overlapping families $\langle \msa^i\in [\mcx_{t_i, N_i}]^{\omega_1}: i<n\rangle$, $\varphi_0(c, \msa^0,..., \msa^{n-1})$ is the assertion that for any collection $\langle s_\alpha: \alpha<\omega_1\rangle$ of pairwise distinct elements of $S$ and any $\langle x^i_\alpha\in [\msa^i]^{<\omega}: i<n, \alpha<\omega_1\rangle$ such that $\langle x^i_\alpha: \alpha<\omega_1\rangle$ is pairwise disjoint for each $i<n$, there are $\alpha<\beta$ such that
\begin{enumerate}[$(i)$]
\item $s_\alpha<_S s_\beta$;
\item for every $i<n$, $a\in x^i_\alpha$ and $b\in x^i_\beta$, there is $j<2$ such that for all $k, k'<N_{i}$,
$$c(a(k), b(k'))=j \text{ and } c(a(N_{i}+k), b(N_{i}+k'))=1-j.$$
\end{enumerate}

\item $\varphi_0(c)$ is the assertion that for $n>0$, $\langle (t_i, N_i)\in S\times \omega\setminus\{0\}: i<n\rangle$ and non-overlapping families $\langle \msa^i\in [\mcx_{t_i, N_i}]^{\omega_1}: i<n\rangle$, $\varphi_0(c, \msa^0,..., \msa^{n-1})$ holds.
\end{enumerate}
\end{defn}
By Theorem \ref{thm precaliber}, $\varphi_0(c)$ holds.
We then introduce a poset that is slightly different from the one in \cite[Definition 6]{PW24}.
\begin{defn}\label{defn non-absolute}
For $0<n<\omega$, $s\in S$ and non-overlapping family $\msa\in [\mcx_{s,n}]^{\omega_1}$, $\mc{Q}_\msa$ is the poset consisting of $p\in [\msa]^{<\omega}$ such that
\begin{itemize}
\item for $a, b$ in $p$ with $a(2n-1)<_S b(2n-1)$, there is $i<2$ satisfying 
\[c(a(j), b(k))=i\text{ and $c(a(n+j), b(n+k))=1-i$ whenever }j,k<n.\] 
\end{itemize}
The order is reverse inclusion.
\end{defn}

$\varphi_0(c)$ asserts that every finite product of posets defined above is ccc. And the conclusion of Theorem \ref{thm precaliber} asserts that posets defined above have precaliber $\omega_1$. Moreover, by the proof of \cite[Theorem 13]{PW24}, to every poset defined above, there is a proper forcing that adds an uncountable antichain.
\begin{cor}\label{cor varphi}
Suppose $S\subseteq 2^{<\omega_1}$ is a complete coherent Suslin tree, $\mcf$ is a 2-thin Kurepa family wrt $S$ and $c: [S]^2\ra 2$ is a coloring satisfying (coh) and (hom). Then $\varphi_0(c)$ holds.
\end{cor}

Recall \cite[Corollary 4]{PW24} that if $\varphi_0(c)$ holds, then $\Vdash_S {\rm Pr}_1(\omega_1,2,\omega)$. In fact, if $b$ is a generic branch of $S$, then in $V[b]$, the coloring  witnessing ${\rm Pr}_1(\omega_1,2,\omega)$ is 
\[\pi: [\omega_1]^2\ra 2 \text{ defined by }\pi(\alpha, \beta)=c(b\up \alpha, b\up \beta).\]
Moreover, properties (coh) and (hom)  induce the following additional properties to the witnessing coloring $\pi$.
\begin{itemize}
\item $\pi$ is coherent.
\item For every uncountable pariwise disjoint family $\msa\subseteq [\omega_1]^{<\omega}$, there exists $\msb\in [\msa]^{\omega_1}$ such that for every $a, b$ in $\msb$, if $\max(a)<\min(b)$, then for some $i<2$, $\pi(a(j), b(k))=i$ for all $j<|a|$ and $k<|b|$.
\end{itemize}
Together with Lemma \ref{lem ccc indestructible} and above corollary, we get  the following conclusion.
\begin{cor}\label{cor indes}
Suppose  there is a thin Kurepa family. If  $\pp$ is a ccc poset and $\Vdash_\pp \dot{S} \subseteq 2^{<\omega_1}$ is a complete coherent Suslin tree, then $\Vdash_{\pp* \dot{S}} {\rm Pr}_1(\omega_1,2,\omega)$. 
\end{cor} 
\begin{proof}
Work in $V[G]$ where $G$ is a generic filter over $\pp$.   Let $S=\dot{S}^G$. By Lemma \ref{lem ccc indestructible}, there is a thin Kurepa family. By Lemma \ref{lem stkf}, there is a 2-thin Kurepa family wrt $S$. By Theorem \ref{thm coloring}, there is a coloring $c: [S]^2\ra 2$ satisfying (coh) and (hom). Then by Corollary \ref{cor varphi}, $\varphi_0(c)$ holds. Finally by \cite[Corollary 4]{PW24}, $\Vdash_S  {\rm Pr}_1(\omega_1,2,\omega)$. 
\end{proof}

Above corollary indicates that in order to get a Suslin extension in which ${\rm Pr}_1(\omega_1,2,\omega)$ fails, non-ccc forcings are likely needed. In fact, it is not difficult to destroy a thin Kurepa family by proper forcings. One way is to simply apply the proof of \cite[Theorem 13]{PW24}: If $\varphi_0(c)$ holds, then there is a proper forcing that preserves Suslin of $S$ and forces an uncountable $F\subseteq S$ such that $c(u,v)=0$ for all $u<_S v$ in $F$. In other words, if $\varphi_0(c)$ holds, then there is a proper forcing that preserves Suslin of $S$ and destroys $\varphi_0(c)$. In particular, the thin Kurepa family used to guarantee the conclusion of Theorem \ref{thm precaliber} is destroyed. Iterating sufficiently many proper posets of this kind with countable support induces a model with no thin Kurepa family. So unlike the non-existence of the Kurepa family,  the non-existence of the thin Kurepa family has no large cardinal strength.

\bibliographystyle{plain}

\end{document}